\documentclass[11pt,oneside]{amsart}

\usepackage{amsmath,ifthen, amsfonts, amssymb,
srcltx, amsopn, color, enumerate, hyperref}
\usepackage[cmtip,arrow]{xy}
\usepackage{pb-diagram, pb-xy}
\usepackage{overpic}
\pdfoutput=1

\usepackage[T1]{fontenc}
\usepackage{libertine}

\usepackage{graphicx}

\newcommand{\showcomments}{yes}
\renewcommand{\showcomments}{no}

\newsavebox{\commentbox}
{\ifthenelse{\equal{\showcomments}{yes}}%
{\footnotemark
        \begin{lrbox}{\commentbox}
        \begin{minipage}[t]{1.25in}\raggedright\sffamily\tiny
        \footnotemark[\arabic{footnote}]}
{\begin{lrbox}{\commentbox}}}%
{\ifthenelse{\equal{\showcomments}{yes}}%
{\end{minipage}\end{lrbox}\marginpar{\usebox{\commentbox}}}
{\end{lrbox}}}

\newtheorem{thm}{Theorem}[section]
\newtheorem{lem}[thm]{Lemma}

\newtheorem{prop}[thm]{Proposition}

\newtheorem{thmi}{Theorem}
\newtheorem{cori}[thmi]{Corollary}

\theoremstyle{definition}
\newtheorem{defn}[thm]{Definition}

\newtheorem{remi}[thmi]{Remark}

\newtheorem{rem}[thm]{Remark}

\newtheorem{claim*}{Claim}

\DeclareMathOperator{\dimension}{dim}

\DeclareMathOperator{\image}{im}

\DeclareMathOperator{\link}{Lk}
\DeclareMathOperator{\cstar}{St}

\newcommand{\neb}{\mathcal N}

\newcommand{\field}[1]{\mathbb{#1}}
\newcommand{\integers}{\ensuremath{\field{Z}}}

\newcommand{\naturals}{\ensuremath{\field{N}}}
\newcommand{\reals}{\ensuremath{\field{R}}}

\newcommand{\closure}[1]{Cl\left({#1}\right)}
\newcommand{\floor}[1]{\ensuremath{\lfloor\widehat{#1}\rfloor}}

\makeatletter

\newcommand{\Rmnum}[1]{\mathbf{{\expandafter\@slowromancap\romannumeral #1@}}}

\makeatother

\let\oldmarginpar\marginpar
\renewcommand\marginpar[1]{\-\oldmarginpar[\raggedleft\footnotesize #1]%
{\raggedright\footnotesize #1}}

\newcounter{enumitemp}

\newcommand{\dist}{\textup{\textsf{d}}}

\newcommand{\sepgrowth}{\mathbf{Sep}}

\newcommand{\gate}[2]{\mathfrak g_{_{#2}}({#1})}

\newcommand{\growth}{\mathrm g\mathrm r}

\usepackage{fullpage}

\begin{document}
\title{Quantifying separability in virtually special groups}
\author[M.F. Hagen]{Mark F. Hagen}
\address{Dept. of Pure Maths and Math. Stat., University of Cambridge, Cambridge, UK}
\email{markfhagen@gmail.com}
\author[P. Patel]{Priyam Patel}
\address{Department of Mathematics, Purdue University, West Lafayette, Indiana, USA}
\email{patel376@purdue.edu}
\date{\today}

\begin{abstract}
We give a new, effective proof of the separability of cubically convex-cocompact subgroups of special groups.  As a consequence, we show that if $G$ is a virtually compact special hyperbolic group, and $Q\leq G$ is a $K$-quasiconvex subgroup, then any $g\in G-Q$ of word-length at most $n$ is separated from $Q$ by a subgroup whose index is polynomial in $n$ and exponential in $K$.  This generalizes a result of Bou-Rabee and the authors on residual finiteness growth \cite{BouRabeeHagenPatel:res_finite_special} and a result of the second author on surface groups \cite{Patel:quantifying}. 
\end{abstract}

\maketitle
\setcounter{tocdepth}{2}

\section*{Introduction}
Early motivation for studying residual finiteness and subgroup separability was a result of the relevance of these properties to decision problems in group theory.  An observation of Dyson~\cite{Dyson} and Mostowski~\cite{Mostowski}, related to earlier ideas of McKinsey~\cite{McKinsey}, states that finitely presented residually finite groups have solvable word problem. The word problem is a special case of the membership problem, i.e. the problem of determining whether a given $g\in G$ belongs to a particular subgroup $H$ of $G$. \emph{Separability} can produce a solution to the membership problem in essentially the same way that a solution to the word problem is provided by residual finiteness (see, e.g., the discussion in~\cite{AschenbrennerFriedlWilton:decision}). A subgroup $H\leq G$ is \emph{separable} in $G$ if, for all $g\in G-H$, there exists $G'\leq_{f.i.}G$ with $H\leq G'$ and $g\not\in G'$.  Producing an upper bound, in terms of $g$ and $H$, on the minimal index of such a subgroup $G'$ is what we mean by \emph{quantifying separability} of $H$ in $G$.  Quantifying separability is related to the membership problem; see Remark~\ref{remi:membership} below.

Recently, separability has played a crucial role in low-dimensional topology, namely in the resolutions of the Virtually Haken and Virtually Fibered conjectures~\cite{Agol:virtual_haken,Wise:QCH}. Its influence in topology is a consequence of the seminal paper of Scott \cite{Scott:LERF}, which establishes a topological reformulation of subgroup separability.  Roughly, Scott's criterion allows one to use separability to promote (appropriately construed) immersions to embeddings in finite covers. In \cite{Agol:virtual_haken}, Agol proved the Virtually Special Conjecture of Wise, an outstanding component of the proofs of the above conjectures. Agol's theorem shows that every word hyperbolic cubical group virtually embeds in a right-angled Artin group (hereafter, \emph{RAAG}). Cubically convex-cocompact subgroups of RAAGs are separable~\cite{HsuWise:separating,Haglund:graph_products} and Agol's theorem demonstrates that word hyperbolic cubical groups inherit this property via the virtual embeddings (separability properties are preserved under passing to subgroups and finite index supergroups). In fact, since quasiconvex subgroups of hyperbolic cubical groups are cubically convex-cocompact~\cite{Haglund:graph_products,SageevWise:cores}, all quasiconvex subgroups of such groups are separable. In this paper, we give a new, effective proof of the separability of cubically convex-cocompact subgroups of special groups.  
Our main technical result is:

\begin{thmi}\label{thmi:main}
Let $\Gamma$ be a simplicial graph and let $Z$ be a compact connected cube complex, based at a $0$--cube $x$, with a based local isometry $Z\rightarrow S_\Gamma$.  For all $g\in A_\Gamma-\pi_1Z$, there is a cube complex $(Y,x)$ such that: 
\begin{enumerate}
 \item $Z\subset Y$;
 \item there is a based local isometry $Y\rightarrow S_\Gamma$ such that $Z\rightarrow S_\Gamma$ factors as $Z\hookrightarrow Y\rightarrow S_\Gamma$;
 \item any closed based path representing $g$ lifts to a non-closed path at $x$ in $Y$;
 \item $|Y^{(0)}|\leq|Z^{(0)}|(|g|+1)$, 
\end{enumerate}
\noindent where $|g|$ is the word length of $g$ with respect to the standard generators. 
\end{thmi}

\noindent Via Haglund-Wise's canonical completion \cite{HaglundWise:special}, Theorem~\ref{thmi:main} provides the following bounds on the separability growth function (defined in Section \ref{sec:background}) of the class of cubically convex-cocompact subgroups of a (virtually) special group.  Roughly, separability growth quantifies separability of all subgroups in a given class.

\begin{cori}\label{cor:qcerf_special}\label{cor:qcerf_special}
Let $G\cong\pi_1X$, with $X$ a compact special cube complex, and let $\mathcal Q_R$ be the class of subgroups represented by compact local isometries to $X$ whose domains have at most $R$ vertices. Then $$\sepgrowth^{\mathcal Q_R}_{G,\mathcal S}(Q,n)\leq PRn$$ for all $Q\in\mathcal Q_R$ and $n\in\naturals$, where the constant $P$ depends only on the generating set $\mathcal S$.  Hence, letting $\mathcal Q'_K$ be the class of subgroups $Q\leq G$ such that the convex hull of $Q\tilde x$ lies in $\neb_{K}(Q\tilde x)$ and $\tilde x\in\widetilde X^{(0)}$, $$\sepgrowth^{\mathcal Q'_K}_{G,\mathcal S}(Q,n)\leq P'\growth_{\widetilde X}(K)n,$$ where $P'$ depends only on $G,\widetilde X,\mathcal S$, and $\growth_{\widetilde X}$ is the growth function of $\widetilde X^{(0)}$.
\end{cori}

\noindent In the hyperbolic case, where cubically convex-cocompactness is equivalent to quasiconvexity, we obtain a bound that is polynomial in the length of the word and exponential in the quasiconvexity constant:

\begin{cori}\label{cori:hyperbolic_sep}\label{cor:hyperbolic_1}
Let $G$ be a group with an index-$J$ special subgroup.  Fixing a word-length $\|-\|_{\mathcal S}$ on $G$, suppose that $(G,\|-\|_{\mathcal S})$ is $\delta$-hyperbolic.  For each $K\geq1$, let $\mathcal Q_K$ be the set of subgroups $Q\leq G$ such that $Q$ is $K$-quasiconvex with respect to $\|-\|_{\mathcal S}$.  Then there exists a constant $P=P(G,\mathcal S)$ such that for all $K\geq0,Q\in\mathcal Q_K$, and $n\geq 0$, $$\sepgrowth^{\mathcal Q_K}_{G,\mathcal S}(Q,n)\leq P\growth_{G}(PK)^{J!}n^{J!},$$ where $\growth_G$ is the growth function of $G$. 
\end{cori}

Corollary~\ref{cori:hyperbolic_sep} says that if $G$ is a hyperbolic cubical group, $Q\leq G$ is $K$-quasiconvex, and $g\in G-Q$, then $g$ is separated from $Q$ by a subgroup of index bounded by a function polynomial in $\|g\|_{\mathcal S}$ and exponential in $K$. 

The above results fit into a larger body of work dedicated to quantifying residual finiteness and subgroup separability of various classes of groups (see, e.g., ~\cite{BK12,BM13,KassabovMatucci:bounding,Buskin:efficient,Patel:quantifying,Patel:thesis,Rivin:geodesics,BouRabee:approx,BouRabeeMcReynolds:char,KozmaThom}). When $G$ is the fundamental group of a hyperbolic surface, compare Corollary~\ref{cori:hyperbolic_sep} to~\cite[Theorem 7.1]{Patel:quantifying}.  Combining various cubulation results with \cite{Agol:virtual_haken}, the groups covered by Corollary~\ref{cori:hyperbolic_sep} include fundamental groups of hyperbolic $3$--manifolds~\cite{BergeronWise,KM09}, hyperbolic Coxeter groups~\cite{HaglundWiseCoxeter}, simple-type arithmetic hyperbolic lattices~\cite{BergeronHaglundWiseSimple}, hyperbolic free-by-cyclic groups~\cite{HagenWise:general}, hyperbolic ascending HNN extensions of free groups with irreducible monodromy~\cite{HagenWise:irreducible}, hyperbolic groups with a quasiconvex hierarchy \cite{Wise:QCH}, $C'(\frac{1}{6})$ small cancellation groups \cite{Wise:small_cancel_cube}, and hence random groups at low enough density \cite{OllivierWise:random}, among many others.

In \cite{BouRabeeHagenPatel:res_finite_special} Bou-Rabee and the authors quantified residual finiteness for virtually special groups, by working in RAAGs and appealing to the fact that upper bounds on residual finiteness growth are inherited by finitely-generated subgroups and finite-index supergroups. Theorem~\ref{thmi:main} generalizes a main theorem of \cite{BouRabeeHagenPatel:res_finite_special}, and accordingly the proof is reminiscent of the one in \cite{BouRabeeHagenPatel:res_finite_special}. However, residual finiteness is equivalent to separability of the trivial subgroup, and thus it is not surprising that quantifying separability for an arbitrary convex-cocompact subgroup of a RAAG entails engagement with a more complex geometric situation. Our techniques thus significantly generalize those of \cite{BouRabeeHagenPatel:res_finite_special}.

\begin{remi}[Membership problem]\label{remi:membership}
If $H$ is a finitely-generated separable subgroup of the finitely-presented group $G$, and one has an upper bound on $\sepgrowth^{\{H\}}_{G,\mathcal S}(|g|)$, for some finite generating set $\mathcal S$ of $G$, then the following procedure decides if $g\in H$: first, enumerate all subgroups of $G$ of index at most $\sepgrowth^{\{H\}}_{G,\mathcal S}(|g|)$ using a finite presentation of $G$.  Second, for each such subgroup, test whether it contains $g$; if so, ignore it, and if not, proceed to the third step.  Third, for each finite-index subgroup not containing $g$, test whether it contains each of the finitely many generators of $H$; if so, we have produced a finite-index subgroup containing $H$ but not $g$, whence $g\not\in H$.  If we exhaust the subgroups of index at most $\sepgrowth^{\{H\}}_{G,\mathcal S}(|g|)$ without finding such a subgroup, then $g\in H$.  In particular, Corollary~\ref{cori:hyperbolic_sep} gives an effective solution to the membership problem for quasiconvex subgroups of hyperbolic cubical groups, though it does not appear to be any more efficient than the more general solution to the membership problem for quasiconvex subgroups of (arbitrary) hyperbolic groups recently given by Kharlampovich--Miasnikov--Weil in~\cite{KharlampovichMiasnikovWeil:membership}.
\end{remi}

The paper is organized as follows: In Section \ref{sec:background}, we define the separability growth of a group with respect to a class $\mathcal Q$ of subgroups, which generalizes the residual finiteness growth introduced in~\cite{BouRabee:quantifying}. We also provide some necessary background on RAAGs and cubical geometry. In Section \ref{sec:consequences}, we discuss corollaries to the main technical result, including Corollary~\ref{cori:hyperbolic_sep}, before concluding with a proof of Theorem~\ref{thmi:main} in Section~\ref{sec:main_proof}.

\subsection*{Acknowledgments}
We thank K. Bou-Rabee, S. Dowdall, F. Haglund, D.B. McReynolds, N. Miller, H. Wilton, and D.T. Wise for helpful discussions about issues related to this paper.  We also thank an anonymous referee for astute corrections. The authors acknowledge travel support from U.S. National Science Foundation grants DMS 1107452, 1107263, 1107367 "RNMS: Geometric Structures and Representation Varieties" (the GEAR Network) and from grant NSF 1045119.  M.F.H. was supported by the National Science Foundation under Grant Number NSF 1045119.

\section{Background}\label{sec:background}

\subsection{Separability growth}\label{sec:quant_res_prop}
Let $G$ be a group generated by a finite set $\mathcal S$ and let $H\leq G$ be a subgroup.  Let $\Omega_H=\{\Delta\leq G:H\leq\Delta\}$, and define a map $D_G^{\Omega_H}:G-H\rightarrow\naturals\cup\{\infty\}$ by $$D_G^{\Omega_H}(g)=\min\{[G:\Delta]:\Delta\in\Omega_H,g\not\in\Delta\}.$$  This is a special case of the notion of a \emph{divisibility function} defined in~\cite{BouRabee:quantifying} and discussed in~\cite{BM13}.  Note that $H$ is a separable subgroup of $G$ if and only if $D_G^{\Omega_H}$ takes only finite values.

The \emph{separability growth} of $G$ with respect to a class $\mathcal Q$ of subgroups is a function $\sepgrowth_{G,\mathcal S}^{\mathcal Q}:\mathcal Q\times\naturals\rightarrow\naturals\cup\{\infty\}$ given by $$\sepgrowth_{G,\mathcal S}^{\mathcal Q}(Q,n)=\max\left\{D_G^{\Omega_Q}(g):g\in G-Q,\|g\|_\mathcal S\leq n\right\}.$$  If $\mathcal Q$ is a class of separable subgroups of $G$, then the separability growth measures the index of the subgroup to which one must pass in order to separate $Q$ from an element of $G-Q$ of length at most $n$. For example, when $G$ is residually finite and $\mathcal Q=\{\{1\}\}$, then $\sepgrowth^{\mathcal Q}_{G,\mathcal S}$ is the residual finiteness growth function.  The following fact is explained in greater generality in~\cite[Section~2]{BouRabeeHagenPatel:res_finite_special}. (In the notation of \cite{BouRabeeHagenPatel:res_finite_special}, $\sepgrowth^{\mathcal Q}_{G,\mathcal S}(Q,n)=\mathrm{RF}^{\Omega_Q}_{G,\mathcal S}(n)$ for all $Q\in\mathcal Q$ and $n\in\naturals$.)

\begin{prop}\label{prop:growth}
Let $G$ be a finitely generated group and let $\mathcal Q$ be a class of subgroups of $G$.  If $\mathcal S, \mathcal S'$ are finite generating sets of $G$, then there exists a constant $C>0$ with $$\sepgrowth_{G,\mathcal S'}^{\mathcal Q}(Q,n)\leq C\cdot\sepgrowth_{G,\mathcal S}^{\mathcal Q}(Q,Cn)$$ for $Q\in\mathcal Q,n\in\naturals$.  Hence the asymptotic growth rate of \,$\sepgrowth^{\mathcal Q}_{G,\mathcal S}$ is independent of $\mathcal S$.
\end{prop}

\noindent(Similar statements assert that upper bounds on separability growth are inherited by finite-index supergroups and arbitrary finitely-generated subgroups but we do not use, and thus omit, these.)

\subsection{Nonpositively-curved cube complexes}\label{sec:npccc}
We assume familiarity with nonpositively-curved and CAT(0) cube complexes and refer the reader to e.g.~\cite{Hagen:QuasiArb,Haglund:graph_products,WiseNotes, Wise:QCH} for background.  We now make explicit some additional notions and terminology, related to convex subcomplexes, which are discussed in greater depth in~\cite{BehrstockHagenSisto:curve_graph}.  We also discuss some basic facts about RAAGs and Salvetti complexes.  Finally, we will use the method of \emph{canonical completion}, introduced in~\cite{HaglundWise:special}, and refer the reader to~\cite[Lemma~2.8]{BouRabeeHagenPatel:res_finite_special} for the exact statement needed here.

\subsubsection{Local isometries, convexity, and gates}\label{sec:convexity}
A \emph{local isometry} $\phi:Y\rightarrow X$ of cube complexes is a locally injective combinatorial map with the property that, if $e_1,\ldots,e_n$ are $1$--cubes of $Y$ all incident to a $0$--cube $y$, and the (necessarily distinct) $1$--cubes $\phi(e_1),\ldots,\phi(e_n)$ all lie in a common $n$--cube $c$ (containing $\phi(y)$), then $e_1,\ldots,e_n$ span an $n$--cube $c'$ in $Y$ with $\phi(c')=c$.  If $\phi:Y\rightarrow X$ is a local isometry and $X$ is nonpositively-curved, then $Y$ is as well.  Moreover, $\phi$ lifts to an embedding $\tilde\phi:\widetilde Y\rightarrow\widetilde X$ of universal covers, and $\tilde\phi(\widetilde Y)$ is \emph{convex} in $\widetilde X$ in the following sense.

Let $\widetilde X$ be a CAT(0) cube complex.  The subcomplex $K\subseteq\widetilde X$ is \emph{full} if $K$ contains each $n$--cube of $\widetilde X$ whose $1$-skeleton appears in $K$.  If $K$ is full, then $K$ is \emph{isometrically embedded} if $K\cap \bigcap_iH_i$ is connected whenever $\{H_i\}$ is a set of pairwise-intersecting hyperplanes of $\widetilde X$.  Equivalently, the inclusion $K^{(1)}\hookrightarrow\widetilde X^{(1)}$ is an isometric embedding with respect to the graph-metric.  If the inclusion $K\hookrightarrow\widetilde X$ of the full subcomplex $K$ is a local isometry, then $K$ is \emph{convex}.  Note that a convex subcomplex is necessarily isometrically embedded, and in fact $K$ is convex if and only if $K^{(1)}$ is metrically convex in $\widetilde X^{(1)}$.  A convex subcomplex $K$ is a CAT(0) cube complex in its own right, and its hyperplanes have the form $H\cap K$, where $K$ is a hyperplane of $\widetilde X$.  Moreover, if $K$ is convex, then hyperplanes $H_1\cap K,H_2\cap K$ of $K$ intersect if and only if $H_1\cap H_2\neq\emptyset$.  We often say that the hyperplane $H$ \emph{crosses} the convex subcomplex $K$ to mean that $H\cap K\neq\emptyset$ and we say the hyperplanes $H,H'$ \emph{cross} if they intersect.

Hyperplanes are an important source of convex subcomplexes, in two related ways.  First, recall that for all hyperplanes $H$ of $\widetilde X$, the carrier $\neb(H)$ is a convex subcomplex.  Second, $\neb(H)\cong H\times[-\frac{1}{2},\frac{1}{2}]$, and the subcomplexes $H\times\{\pm\frac{1}{2}\}$ of $\widetilde X$ ``bounding'' $\neb(H)$ are convex subcomplexes isomorphic to $H$ (when $H$ is given the cubical structure in which its $n$--cubes are midcubes of $(n+1)$--cubes of $\widetilde X$).  A subcomplex of the form $H\times\{\pm\frac{1}{2}\}$ is a \emph{combinatorial hyperplane}. The \emph{convex hull} of a subcomplex $\mathcal S\subset\widetilde X$ is the intersection of all convex subcomplexes that contain $\mathcal S$; see~\cite{Haglund:graph_products}.

Let $K\subseteq\widetilde X$ be a convex subcomplex.  Then there is a map $\mathfrak g_K:\widetilde X^{(0)}\rightarrow K$ such that for all $x\in\widetilde X^{(0)}$, the point $\gate{x}{K}$ is the unique closest point of $K$ to $x$.  (This point is often called the \emph{gate} of $x$ in $K$; gates are discussed further in \cite{Chepoi:median} and \cite{BandeltChepoi_survey}.)  This map extends to a cubical map $\mathfrak g_{K}:\widetilde X\rightarrow K$, the \emph{gate map}.  See e.g.~\cite{BehrstockHagenSisto:curve_graph} for a detailed discussion of the gate map in the language used here; we use only that it extends the map on $0$--cubes and has the property that for all $x,y$, if $\gate{x}{K},\gate{y}{K}$ are separated by a hyperplane $H$, then the same $H$ separates $x$ from $y$.  Finally, the hyperplane $H$ separates $x$ from $\gate{x}{K}$ if and only if $H$ separates $x$ from $K$.  The gate map allows us to define the \emph{projection} of the convex subcomplex $K'$ of $\widetilde X$ onto $K$ to be $\gate{K}{K'}$, which is the convex hull of the set $\{\gate{x}{K} \in K : x \in K'^{(0)} \}$.  Convex subcomplexes $K,K'$ are \emph{parallel} if $\gate{K}{K'}=K'$ and $\gate{K'}{K}=K$.  Equivalently, $K,K'$ are \emph{parallel} if and only if, for each hyperplane $H$, we have $H\cap K\neq\emptyset$ if and only if $H\cap K'\neq\emptyset$.  Note that parallel subcomplexes are isomorphic.

\begin{rem}\label{rem:structure_of_convex_hulls_of_pairs}
We often use the following facts.  Let $K,K'$ be convex subcomplexes of $\widetilde X$.  Then the convex hull $C$ of $K\cup K'$ contains the union of $K,K'$ and a convex subcomplex of the form $G_K(K')\times\hat\gamma$, where $G_K(K')$ is the image of the gate map discussed above and $\hat\gamma$ is the convex hull of a geodesic segment $\gamma$ joining a closest pair of $0$--cubes in $K,K'$, by~\cite[Lemma~2.4]{BehrstockHagenSisto:curve_graph}.  A hyperplane $H$ crosses $K$ and $K'$ if and only if $H$ crosses $G_K(K')$; the hyperplane $H$ separates $K,K'$ if and only if $H$ crosses $\hat\gamma$.  All remaining hyperplanes either cross exactly one of $K,K'$ or fail to cross $C$.  Observe that the set of hyperplanes separating $K,K'$ contains no triple $H,H',H''$ of disjoint hyperplanes, none of which separates the other two.  (Such a configuration is called a \emph{facing triple}.)
\end{rem}

\subsubsection{Salvetti complexes and special cube complexes}\label{sec:salvetti}
Let $\Gamma$ be a simplicial graph and let $A_\Gamma$ be the corresponding right-angled Artin group (RAAG), i.e. the group presented by $$\big\langle V(\Gamma)\mid [v,w],\,\{v,w\}\in E(\Gamma)\big\rangle,$$ where $V(\Gamma)$ and $E(\Gamma)$ respectively denote the vertex- and edge-sets of $\Gamma$.  The phrase \emph{generator of $\Gamma$} refers to this presentation; we denote each generator of $A_\Gamma$ by the corresponding vertex of $\Gamma$.  

The RAAG $A_\Gamma$ is isomorphic to the fundamental group of the \emph{Salvetti complex} $S_\Gamma$, introduced in~\cite{CharneyDavis:salvetti_cubes}, which is a nonpositively-curved cube complex with one $0$--cube $x$, an oriented $1$--cube for each $v\in V(\Gamma)$, labeled by $v$, and an $n$--torus (an $n$--cube with opposite faces identified) for every $n$--clique in $\Gamma$.  

A cube complex $X$ is \emph{special} if there exists a simplicial graph $\Gamma$ and a local isometry $X\rightarrow S_\Gamma$ inducing a monomorphism $\pi_1X\rightarrow A_\Gamma$ and a $\pi_1X$-equivariant embedding $\widetilde X\rightarrow\widetilde S_\Gamma$ of universal covers whose image is a convex subcomplex.  Specialness allows one to study geometric features of $\pi_1X$ by working inside of $\widetilde S_\Gamma$, which has useful structure not necessarily present in general CAT(0) cube complexes; see Section~\ref{sec:frames}.  Following Haglund-Wise~\cite{HaglundWise:special}, a group $G$ is \emph{(virtually) [compact] special} if $G$ is (virtually) isomorphic to the fundamental group of a [compact] special cube complex.

\subsubsection{Cubical features particular to Salvetti complexes}\label{sec:frames}
Let $\Gamma$ be a finite simplicial graph and let $\Lambda$ be an induced subgraph of $\Gamma$.  The inclusion $\Lambda\hookrightarrow\Gamma$ induces a monomorphism $A_\Lambda\rightarrow A_\Gamma$.  In fact, there is an injective local isometry $S_\Lambda\rightarrow S_\Gamma$ inducing $A_\Lambda\rightarrow A_\Gamma$.  Hence each conjugate $A_\Lambda^g$ of $A_\Lambda$ in $A_\Gamma$ is the stabilizer of a convex subcomplex $g\widetilde S_\Lambda\subseteq\widetilde S_\Gamma$.  A few special cases warrant extra consideration.

When $\Lambda\subset\Gamma$ is an $n$-clique, for some $n\geq 1$, then $S_\Lambda\subseteq S_\Gamma$ is an $n$--torus, which is the Salvetti complex of the sub-RAAG isomorphic to $\integers^n$ generated by $n$ pairwise-commuting generators.  In this case, $S_\Lambda$ is a \emph{standard $n$-torus} in $S_\Gamma$.  (When $n=1$, $S_\Lambda$ is a \emph{standard circle}.)  Each lift of $\widetilde S_\Lambda$ to $\widetilde S_\Gamma$ is a \emph{standard flat}; when $n=1$, we use the term \emph{standard line}; a compact connected subcomplex of a standard line is a \emph{standard segment}.  The labels and orientations of $1$--cubes in $S_\Gamma$ pull back to $\widetilde S_\Gamma$; a standard line is a convex subcomplex isometric to $\reals$, all of whose $1$--cubes have the same label, such that each $0$--cube has one incoming and one outgoing $1$--cube.

When $\link(v)$ is the link of a vertex $v$ of $\Gamma$, the subcomplex $S_{\link(v)}$ is an immersed combinatorial hyperplane in the sense that $\widetilde S_{\link(v)}$ is a combinatorial hyperplane of $\widetilde S_\Gamma$.  There is a corresponding hyperplane, whose carrier is bounded by $\widetilde S_{\link(v)}$ and $v\widetilde S_{\link(v)}$, that intersects only $1$--cubes labeled by $v$.  Moreover, $\widetilde S_{\link v}$ is contained in $\widetilde S_{\cstar(v)}$, where $\cstar(v)$ is the star of $v$, i.e. the join of $v$ and ${\link(v)}$.  It follows that $\widetilde S_{\cstar(v)}\cong\widetilde S_{\link(v)}\times \widetilde S_v$, where $\widetilde S_v$ is a standard line.  Note that the combinatorial hyperplane $\widetilde S_{\link(v)}$ is parallel to $v^k\widetilde S_{\link(v)}$ for all $k\in\integers$.  Likewise, $\widetilde S_v$ is parallel to $g\widetilde S_v$ exactly when $g\in A_\Lambda$, and parallel standard lines have the same labels.  We say $\widetilde S_v$ is a \emph{standard line dual to $\widetilde S_{\link(v)}$}, and is a standard line dual to any hyperplane $H$ such that $N(H)$ has $\widetilde S_{\link(v)}$ as one of its bounding combinatorial hyperplanes.    

\begin{rem}
We warn the reader that a given combinatorial hyperplane may correspond to distinct hyperplanes whose dual standard lines have different labels; this occurs exactly when there exist multiple vertices in $\Gamma$ whose links are the same subgraph.  However, the standard line dual to a genuine (non-combinatorial) hyperplane is uniquely-determined up to parallelism.
\end{rem}

\begin{defn}[Frame]\label{defn:frame}
Let $K\subseteq\widetilde S_\Gamma$ be a convex subcomplex and let $H$ be a hyperplane.  Let $L$ be a standard line dual to $H$.  The \emph{frame} of $H$ is the convex subcomplex $H'\times L\subseteq\widetilde S_\Gamma$ described above, where $H'$ is a combinatorial hyperplane bounding $N(H)$.  If $K\subseteq\widetilde S_\Gamma$ is a convex subcomplex, and $H$ intersects $K$, then the \emph{frame of $H$ in $K$} is the complex $K\cap(H\times L)$.  It is shown in~\cite{BouRabeeHagenPatel:res_finite_special} that the frame of $H$ in $K$ has the form $(H\cap K)\times(L\cap K)$, provided that $L$ is chosen in its parallelism class to intersect $K$.  Note that the frame of $H$ is in fact well-defined, since all possible choices of $L$ are parallel.
\end{defn}

\section{Consequences of Theorem~\ref{thmi:main}}\label{sec:consequences}
Assuming Theorem~\ref{thmi:main}, we quantify separability of cubically convex-cocompact subgroups of special groups with the proofs of Corollaries~\ref{cor:qcerf_special} and \ref{cor:hyperbolic_1}, before proving Theorem~\ref{thmi:main} in the next section.

\begin{proof}[Proof of Corollary~\ref{cor:qcerf_special}]
Let $\Gamma$ be a finite simplicial graph so that there is a local isometry $X\rightarrow S_\Gamma$.  Let $Q\in\mathcal Q_R$ be represented by a local isometry $Z\rightarrow X$.  Then for all $g\in\pi_1X-\pi_1Z$, by Theorem~\ref{thmi:main}, there is a local isometry $Y\rightarrow S_\Gamma$ such that $Y$ contains $Z$ as a locally convex subcomplex, and $g\not\in\pi_1Y$, and $|Y^{(0)}|\leq|Z^{(0)}|(|g|+1)$.  Applying canonical completion~\cite{HaglundWise:special} to $Y\rightarrow S_\Gamma$ yields a cover $\widehat S_\Gamma\rightarrow S_\Gamma$ in which $Y$ embeds; this cover has degree $|Y^{(0)}|$ by ~\cite[Lemma~2.8]{BouRabeeHagenPatel:res_finite_special}. Let $H'=\pi_1\widehat S_\Gamma\cap\pi_1 X$, so that $\pi_1 Z\leq H'$, and $g\not\in H'$, and $[\pi_1X:H']\leq|Z^{(0)}|(|g|+1)$.  The first claim follows.  

Let $G\cong\pi_1X$ with $X$ compact special, $Q\leq G$, and the convex hull of $Q\tilde x$ in $\widetilde X$ lies in $\neb_K(Q\tilde X)$. Then the second claim follows since we can choose $Z$ to be the quotient of the hull of $Q\tilde x$ by the action of $Q$, and $|Z^{(0)}|\leq\growth_{\widetilde X}(K)$.
\end{proof}

In general, the number of 0--cubes in $Z$ is computable from the quasiconvexity constant of a $Q$-orbit in $\widetilde X^{(1)}$ by~\cite[Theorem~2.28]{Haglund:graph_products}).  In the hyperbolic case, we obtain Corollary~\ref{cor:hyperbolic_1} in terms of the quasiconvexity constant, without reference to any particular cube complex:

\begin{proof}[Proof of Corollary~\ref{cor:hyperbolic_1}]
We use Corollary~\ref{cor:qcerf_special} when $J=1$, and promote the result to a polynomial bound when $J\geq 1$. Let $Q\in\mathcal Q_K$ and let $g \in G - Q$.

\emph{The special case:}  Suppose $J=1$ and let $X$ be a compact special cube complex with $G\cong\pi_1X$. Let $Z\rightarrow X$ be a compact local isometry representing the inclusion $Q\rightarrow G$.  Such a complex exists by quasiconvexity of $Q$ and~\cite[Theorem~2.28]{Haglund:graph_products}, although we shall use the slightly more computationally explicit proof in~\cite{SageevWise:cores}.  Let $A'\geq1,B'\geq0$ be constants such that an orbit map $(G,\|-\|_{\mathcal S})\rightarrow(\widetilde X^{(1)},\dist)$ is an $(A',B')$-quasi-isometric embedding, where $\dist$ is the graph-metric.  Then there exist constants $A,B$, depending only on $A',B'$ and hence on $\|-\|_{\mathcal S}$, such that $Qx$ is $(AK+B)$-quasiconvex, where $x$ is a $0$--cube in $\widetilde Z\subset\widetilde X$.  By the proof of Proposition~3.3 of~\cite{SageevWise:cores}, the convex hull $\widetilde Z$ of $Qx$ lies in the $\rho$-neighborhood of $Qx$, where $\rho=AK+B+\sqrt{\dimension X}+\delta'\left(\csc\left(\frac{1}{2}\sin^{-1}(\frac{1}{\sqrt{\dimension X}}\right)+1\right)$ and $\delta'=\delta'(\delta,A',B')$.  Corollary~\ref{cor:qcerf_special} provides $G'\leq G$ with $g\not\in G'$, and $[G:G']\leq|Z^{(0)}|(|g|+1)$.  But $|g|+1\leq A'\|g\|_{\mathcal S}+B'+1$, while $|Z^{(0)}|\leq \growth_{\widetilde X}(\rho)$.  Thus $[G:G']\leq \growth_{\widetilde X}(\rho)A'\|g\|_{\mathcal S}+\growth_{\widetilde X}(\rho)B'+\growth_{\widetilde X}(\rho)$, so that there exists $P_1$ such that $$\sepgrowth^{\mathcal Q_K}_{G,\mathcal S}(Q,n)\leq P_1\growth_{\widetilde X}(P_1K) n$$ for all $K,Q\in\mathcal Q_K,n\in\naturals$, where $P_1$ depends only on $X$.  

\emph{The virtually special case:}  Now suppose that $J\geq 1$. We have a compact special cube complex $X$, and $[G:G']\leq J!$, where $G'\cong\pi_1X$ and $G'\triangleleft G$.  Let $Q\leq G$ be a $K$-quasiconvex subgroup. By Lemma~\ref{lem:quasiconvexity_intersection}, there exists $C=C(G,\mathcal S)$ such that $Q\cap G'$ is $CJ!(K+1)$-quasiconvex in $G$, and thus is $P_2CJ!(K+1)$-quasiconvex in $G'$, where $P_2$ depends only on $G$ and $\mathcal{S}$.

Let $g\in G-Q$.  Since $G'\triangleleft G$, the product $QG'$ is a subgroup of $G$ of index at most $J!$ that contains $Q$.  Hence, if $g\not\in QG'$, then we are done.  We thus assume $g\in QG'$.  Hence we can choose a left transversal $\{q_1,\ldots,q_s\}$ for $Q\cap G'$ in $Q$, with $s\leq J!$ and $q_1=1$. Write $g=q_ig'$ for some $i \leq s$, with $g'\in G'$.  Suppose that we have chosen each $q_i$ to minimize $\|q_i\|_{\mathcal S}$ among all elements of $q_i(Q\cap G')$, so that, by Lemma~\ref{lem:short_transversal}, $\|q_i\|\leq J!$ for all $i$.  Hence $\|g'\|_{\mathcal S}\leq (\|g\|_{\mathcal S}+J!)$.

By the first part of the proof, there exists a constant $P_1$, depending only on $G,G',\mathcal S$, and a subgroup $G''\leq G'$ such that $Q\cap G'\leq G''$, and $g'\not\in G''$, and 
$$[G':G'']\leq P_1\growth_{G'}(P_1P_2CJ!(K+1))\|g'\|_{\mathcal S} \leq P_1\growth_G(P_1P_2CJ!(K+1))\|g'\|_{\mathcal S}.$$

Let $G'''=\cap_{i=1}^sq_iG''q_i^{-1}$, so that $g'\not\in G'''$, and $Q\cap G'\leq G'''$ (since $G'$ is normal), and 
$$[G':G''']\leq (P_1\growth_G(P_1P_2CJ!(K+1))\|g'\|_{\mathcal S})^s.$$
Finally, let $H=QG'''$.  This subgroup clearly contains $Q$.  Suppose that $g=q_ig'\in H$.  Then $g'\in QG'''$, i.e. $g'=ag'''$ for some $a\in Q$ and $g'''\in G'''$.  Since $g'\in G'$ and $G'''\leq G'$, we have $a\in Q\cap G'$, whence $a\in G'''$, by construction.  This implies that $g'\in G'''\leq G''$, a contradiction.  Hence $H$ is a subgroup of $G$ separating $g$ from $Q$.  Finally, $$[G:H]\leq[G:G''']\leq J!\left[P_1\growth_G(P_1P_2CJ!(K+1))(\|g\|_{\mathcal S}+J!)\right]^{J!},$$
and the proof is complete.
\end{proof}

\begin{lem}\label{lem:quasiconvexity_intersection}
Let the group $G$ be generated by a finite set $\mathcal S$ and let $(G,\|-\|_{\mathcal S})$ be $\delta$-hyperbolic.  Let $Q\leq G$ be $K$-quasiconvex, and let $G'\leq G$ be an index-$I$ subgroup.  Then $Q\cap G'$ is $CI(K+1)$-quasiconvex in $G$ for some $C$ depending only on $\delta$ and $\mathcal S$.
\end{lem}

\begin{proof}
Since $Q$ is $K$-quasiconvex in $G$, it is generated by a set $\mathcal T$ of $q\in Q$ with $\|q\|_{\mathcal S}\leq 2K+1$, by~\cite[Lemma~III.$\Gamma$.3.5]{BridsonHaefliger}.  A standard argument shows $(Q,\|-\|_{\mathcal T})\hookrightarrow(G,\|-\|_{\mathcal S})$ is a $(2K+1,0)$-quasi-isometric embedding.  Lemma~\ref{lem:short_transversal} shows that $Q\cap G'$ is $I$-quasiconvex in $(Q,\|-\|_{\mathcal T})$, since $[Q:Q\cap G']\leq I$.  Hence $Q\cap G'$ has a generating set making it $((2I+1)(2K+1),0)$-quasi-isometrically embedded in $(G,\|-\|_{\mathcal S})$.  Apply Lemma~\ref{lem:geodesics_fellowtravel} to conclude.
\end{proof}

The following lemma is standard, but we include it to highlight the exact constants involved:

\begin{lem}\label{lem:geodesics_fellowtravel}
Let $G$ be a group generated by a finite set $\mathcal S$ and suppose that $(G,\|-\|_{\mathcal S})$ is $\delta$-hyperbolic.  Then there exists a (sub)linear function $f:\naturals\rightarrow\naturals$, depending on $\mathcal S$ and $\delta$, such that $\sigma\subseteq\neb_{f(\lambda)}(\gamma)$ whenever $\gamma:[0,L]\rightarrow G$ is a $(\lambda,0)$--quasigeodesic and $\sigma$ is a geodesic joining $\gamma(0)$ to $\gamma(L)$.
\end{lem}

\begin{proof}
See e.g. the proof of~\cite[Theorem~III.H.1.7]{BridsonHaefliger}.
\end{proof}

\begin{lem}\label{lem:short_transversal}
Let $Q$ be a group generated by a finite set $\mathcal S$ and let $Q'\leq Q$ be a subgroup with $[Q:Q']=s<\infty$.  Then there exists a left transversal $\{q_1,\ldots,q_{s}\}$ for $Q'$ such that $\|q_i\|_{\mathcal S}\leq s$ for $1\leq i\leq s$.  Hence $Q'$ is $s$-quasiconvex in $Q$.
\end{lem}

\begin{proof}
Suppose that $q_k=s_{i_k}\cdots s_{i_1}$ is a geodesic word in $\mathcal S\cup\mathcal S^{-1}$ and that $q_k$ is a shortest representative of $q_k Q'$. Let $q_j = s_{i_j} \cdots s_{i_1}$ be the word in $Q$ consisting of the last $j$ letters of $q_k$ for all $1 < j < k$, and let $q_1 = 1$. We claim that each $q_j$ is a shortest representative for $q_j Q'$. Otherwise, there exists $p$ with $\|p\|_{\mathcal S} < j$ such that $q_j Q'$ = $p Q'$. But then $s_k \cdots s_{j+1} p Q' = q_k Q'$, and thus $q_k$ was not a shortest representative. It also follows immediately that $q_j Q' \neq q_{j'} Q'$ for $j \neq j'$. Thus, $q_1, q_2, \dots, q_k$ represent distinct left cosets of $Q'$ provided $k \leq s$, and the claim follows. 
\end{proof}

\begin{rem}[Embeddings in finite covers]\label{rem:lift_to_embeddings}
Given a compact special cube complex $X$ and a compact local isometry $Z\rightarrow X$, Theorem~\ref{thmi:main} gives an upper bound on the minimal degree of a finite cover in which $Z$ embeds; indeed, producing such an embedding entails separating $\pi_1Z$ from finitely many elements in $\pi_1X$.  However, it is observed in~\cite[Lemma~2.8]{BouRabeeHagenPatel:res_finite_special} the Haglund--Wise canonical completion construction~\cite{HaglundWise:special} produces a cover $\widehat X\rightarrow X$ of degree $|Z^{(0)}|$ in which $Z$ embeds.  

\end{rem}

\section{Proof of Theorem~\ref{thmi:main}}\label{sec:main_proof}
In this section, we give a proof of the main technical result.

\begin{defn}\label{rem:parallel}
Let $S_\Gamma$ be a Salvetti complex and let $\widetilde S_\Gamma$ be
its universal cover.  The hyperplanes $H,H'$ of $\widetilde S_\Gamma$
are \emph{collateral} if they have a common dual standard line (equivalently, the same frame). Clearly collateralism is an equivalence relation, and collateral hyperplanes are isomorphic and have the same stabilizer.
\end{defn}

 Being collateral implies that the combinatorial hyperplanes bounding the carrier of $H$ are parallel to those bounding the carrier of $H'$. However, the converse is not true when $\Gamma$ contains
multiple vertices whose links coincide.  In the proof of
Theorem~\ref{thmi:main}, we will always work with hyperplanes, rather than combinatorial hyperplanes, unless we explicitly state that
we are referring to combinatorial hyperplanes.  

\begin{proof}[Proof of Theorem~\ref{thmi:main}]
Let $\tilde x\in\widetilde S_\Gamma$ be a lift of the base $0$--cube $x$ in $S_\Gamma$, and let $\widetilde Z\subseteq\widetilde S_\Gamma$ be the lift of the universal cover of $Z$ containing $\tilde x$. 
Since $Z\rightarrow S_\Gamma$ is a local isometry, $\widetilde Z$ is convex. Let $\widehat Z\subset\widetilde Z$ be the convex hull of a compact connected fundamental domain for the action of $\pi_1Z\leq A_\Gamma$ on $\widetilde Z$.  Denote by $K$ the convex hull of $\widehat Z\cup\{g\tilde x\}$ and let $\mathfrak S$ be the set of hyperplanes of $\widetilde S_\Gamma$ intersecting $K$.  We will form a quotient of $K$, restricting to $\widehat Z \rightarrow Z$ on $\widehat Z$, whose image admits a local isometry to $S_\Gamma$.

\emph{The subcomplex $\floor Z$:} Let $\mathfrak L$ be the collection of standard segments $\ell$ in $K$ that map to standard circles in $S_\Gamma$ with the property that $\ell \cap \widehat Z$ has non-contractible image in $Z$. Let $\floor Z$ be convex hull of $\widehat Z\cup\bigcup_{\ell\in\mathfrak L}\ell$, so that $\widehat Z\subseteq\floor Z\subseteq K$.

\emph{Partitioning $\mathfrak S$:}  We now partition $\mathfrak S$ according to the various types of frames in $K$.  First, let $\mathfrak Z$ be the set of hyperplanes intersecting $\widehat Z$. Second, let $\mathfrak N$ be the set of $N\in\mathfrak S - \mathfrak Z$ such that the frame $(N \cap K) \times(L\cap K)$ of $N$ in $K$ has the property that for some choice of $x_0 \in N^{(0)}$, the segment $(\{x_0\} \times L) \cap \widehat Z$ maps to a nontrivial cycle of $1$--cubes in $Z$. Let $n_N \geq 1$ be the length of that cycle.  By convexity of $\widehat Z$, the number $n_N$ is independent of the choice of the segment $L$ within its parallelism class. Note that $\mathfrak N$ is the set of hyperplanes that cross $\floor Z$, but do not cross $\widehat Z$. Hence each $N\in\mathfrak N$ is collateral to some $W\in\mathfrak Z$.  Third, fix a collection $\{H_1,\ldots H_k\}\subset\mathfrak S - \mathfrak Z$ such that:
\begin{enumerate}
 \item \label{item:separation}For $1\leq i\leq k-1$, the hyperplane $H_i$ separates $H_{i+1}$ from $\floor Z$.
 \item \label{item:parallel}For $1\leq i<j\leq k$, if a hyperplane $H$ separates $H_i$ from $H_j$, then $H$ is collateral to $H_\ell$ for some $\ell\in[i,j]$.  Similarly, if $H$ separates $H_1$ from $\floor Z$, then $H$ is collateral to $H_1$, and if $H$ separates $H_k$ from $g\tilde x$, then $H$ is collateral to $H_k$.
 
\item \label{item:pink parallel} For each $i$, the frame $(H_i \cap K) \times L_i$ of $H_i$ in $K$ has the property that for every $h \in H_i^{(0)}$, the image in $Z$ of the segment $(\{h\} \times L_i) \cap \widehat Z$ is empty or contractible.  (Here, $L_i$ is a standard segment of a standard line dual to $H_i$.)
\end{enumerate}

Let $\mathfrak H$ be the set of all hyperplanes of $\mathfrak S - \mathfrak Z$ that are collateral to $H_i$ for some $i$. Condition (3) above ensures that $\mathfrak H \cap \mathfrak N =  \emptyset$, while $\mathfrak H=\emptyset$ only if $K=\floor Z$. Finally, let $\mathfrak B = \mathfrak S - (\mathfrak Z \cup \mathfrak N \cup \mathfrak H)$.  Note that each $B\in\mathfrak B$ crosses some $H_i$. Figure~\ref{fig:hyperplane_families} shows a possible $K$ and various families of hyperplanes crossing it. 

\begin{figure}[h]
\begin{overpic}[trim = 0in 3.2in 0in 3.25in, clip=true, totalheight=0.27\textheight]{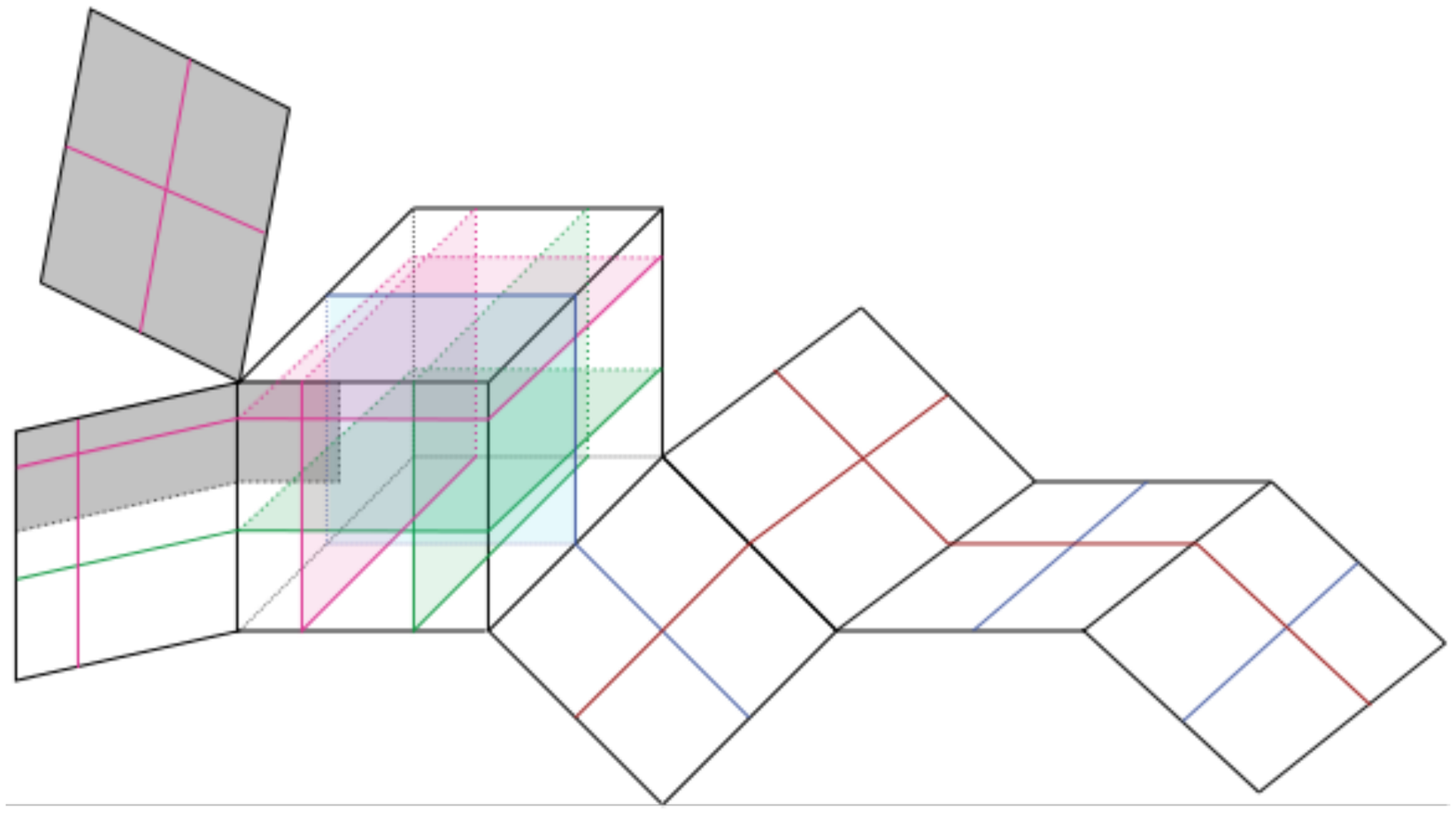}
\put(-0.5,14){\scriptsize{$N_1$}}
\put(39,41){\scriptsize{$N_2$}}
\put(52,4.5){\scriptsize{$H_1$}}
\put(76.5,23){\scriptsize{$H_2$}}
\put(92,16.5){\scriptsize{$H_3$}}
\put(37,4){\scriptsize{$B_1$}}
\put(51,30){\scriptsize{$B_2$}}
\end{overpic}
\caption{Hyperplanes crossing $K$ (the dark shaded area on the left is $\widehat{Z}$). }\label{fig:hyperplane_families}
\end{figure}

\emph{Mapping $\floor Z$ to $Z$:}  We now define a quotient map $\mathfrak q:\floor Z\rightarrow Z$ extending the restriction $\widehat Z\rightarrow Z$ of $\widetilde Z\rightarrow Z$.  Note that if $\mathfrak N=\emptyset$, then $\floor Z=\widehat Z$, and $\mathfrak q$ is just the map $\widehat Z\rightarrow Z$.  Hence suppose $\mathfrak N\neq\emptyset$ and let $\mathfrak N_1,\ldots,\mathfrak N_s$ be the collateralism classes of hyperplanes in $\mathfrak N$, and for $1\leq i\leq s$, let $\mathfrak N'_i$ be the collateralism class of $\mathfrak N_i$ in $\mathfrak S$, i.e. $\mathfrak N_i$ together with a nonempty set of collateral hyperplanes in $\mathfrak Z$.  For each $i$, let $L_i$ be a maximal standard line segment of $\floor Z$, each of whose $1$--cubes is dual to a hyperplane in $\mathfrak N_i'$ and which crosses each element of $\mathfrak N'_i$. For each $i$, let $N_i\in\mathfrak N_i$ be a hyperplane separating $\widehat Z$ from $g\tilde x$.  Then $N_i\cap N_j\neq\emptyset$ for $i\neq j$, since neither separates the other from $\widehat Z$. We can choose the $L_i$ so that there is an isometric embedding $\prod_{i=1}^kL_i\rightarrow\floor Z$, since whether or not two hyperplanes of $\widetilde S_\Gamma$ cross depends only on their collateralism classes.

For each nonempty $I\subseteq\{1,\ldots,k\}$, a hyperplane $W \in\mathfrak Z$ crosses some $U\in\cup_{i\in I}\mathfrak N'_i$ if and only if $W$ crosses each hyperplane collateral to $U$.  Hence, by~\cite[Lemma~7.11]{Hagen:QuasiArb}, there is a maximal convex subcomplex $Y(I)\subset\widehat Z$, defined up to parallelism, such that a hyperplane $W$ crosses each $U\in\cup_{i\in I}\mathfrak N'_i$ if and only if $W\cap Y(I)\neq\emptyset$.  Let $\mathfrak A(I)$ be the set of hyperplanes crossing $Y(I)$.  By the definition of $Y(I)$ and~\cite[Lemma~7.11]{Hagen:QuasiArb}, there is a combinatorial isometric embedding $Y(I)\times\prod_{i\in I}L_i\rightarrow\floor Z$, whose image we denote by $F(I)$ and refer to as a \emph{generalized frame}.  Moreover, for any $0$--cube $z\in\floor Z$ that is not separated from a hyperplane in $\cup_{i\in I}\mathfrak N'_i\cup \mathfrak A(I)$ by a hyperplane not in that set, we can choose $F(I)$ to contain $z$ (this follows from the proof of~\cite[Lemma~7.11]{Hagen:QuasiArb}). Figures~\ref{fig:Nhyperplanes}, \ref{fig:Y(I)1}, \ref{fig:Y(I)2}, and \ref{fig:Y(I)1,2} show possible $\mathfrak{N_i'}$'s and generalized hyperplane frames.

\begin{figure}[h]
\begin{minipage}[b]{0.45\textwidth}
\centering
\begin{overpic}[trim = 0in 3in 3.55in 3in, clip=true, totalheight=0.19\textheight]{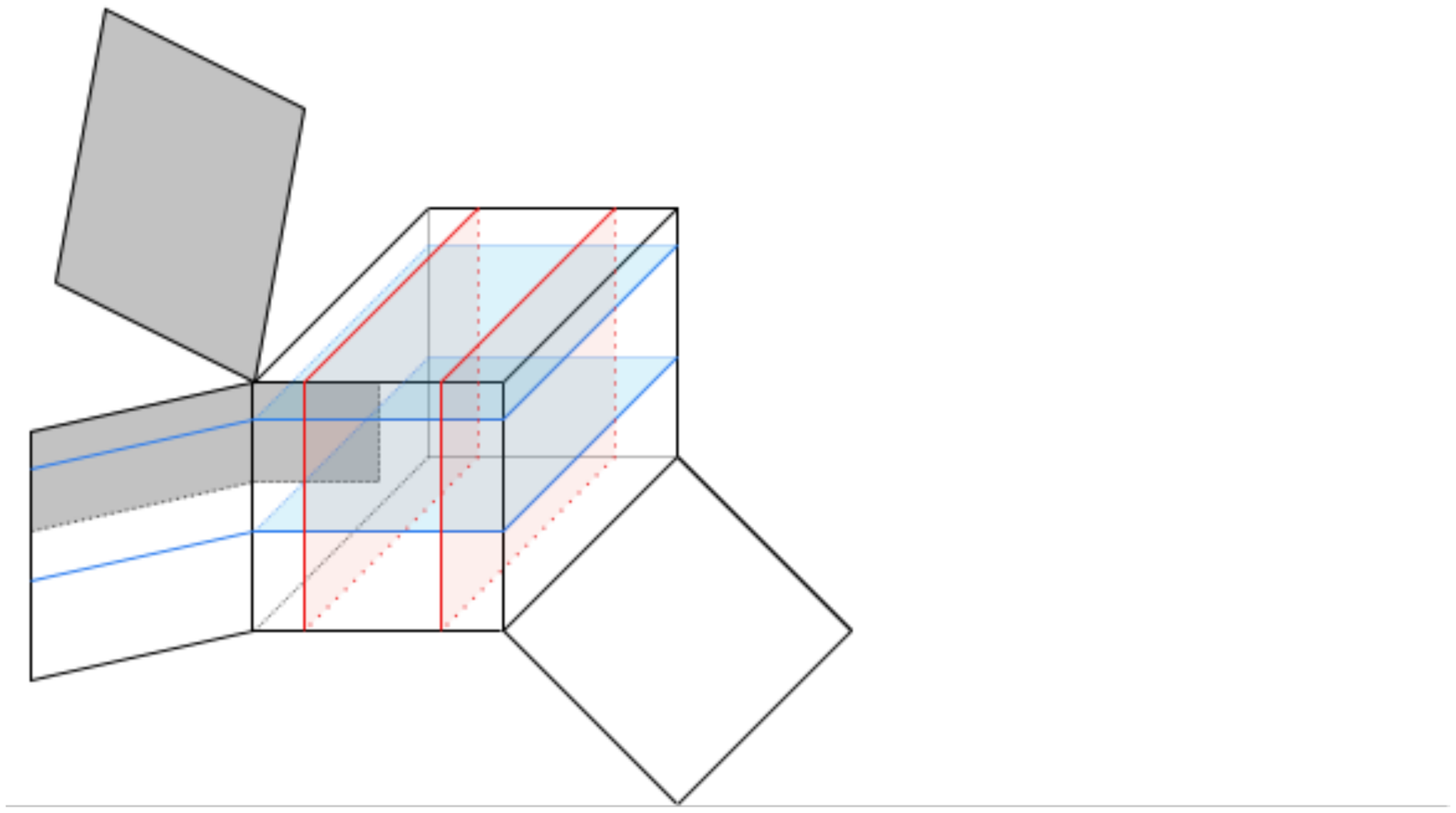}
\put(53,74){\scriptsize{$\mathfrak{N'_1}$}}
\put(-2,25){\scriptsize{$\mathfrak{N'_2}$}}
\end{overpic}
\caption{Collateral families $\mathfrak{N_1'}$ and $\mathfrak{N_2'}$.}\label{fig:Nhyperplanes}
\end{minipage}
\begin{minipage}[b]{.45\textwidth}
\centering
\begin{overpic}[trim = 0in 2.8in 3.55in 3.345in, clip=true, totalheight=0.2\textheight]{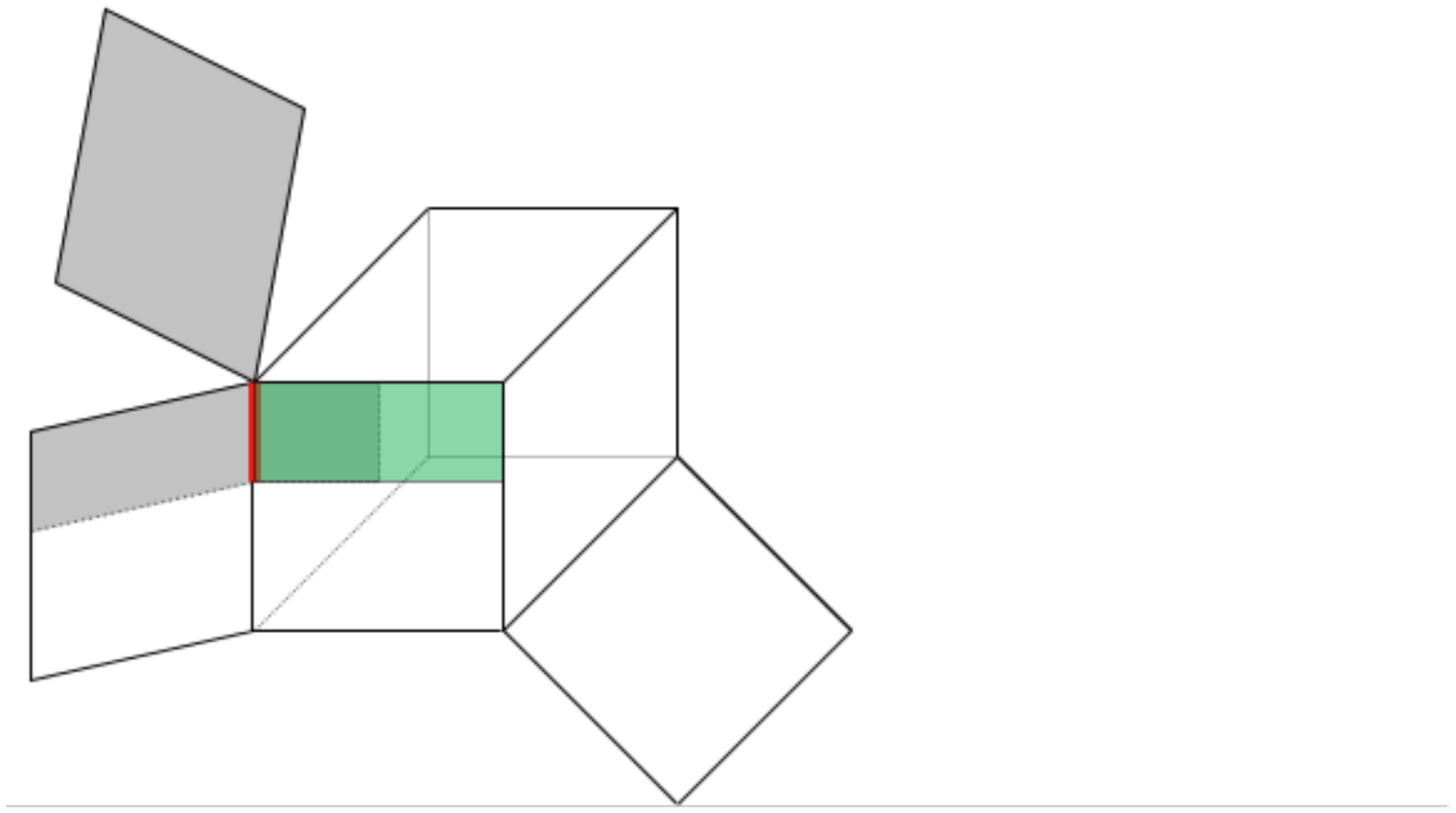}
\end{overpic}
\caption{$Y(\{1\}) \times L_1$.}\label{fig:Y(I)1}
\end{minipage}
\begin{minipage}[b]{.45\textwidth}
\centering
\begin{overpic}[trim = 0in 3in 3.55in 2.8in, clip=true, totalheight=0.2\textheight]{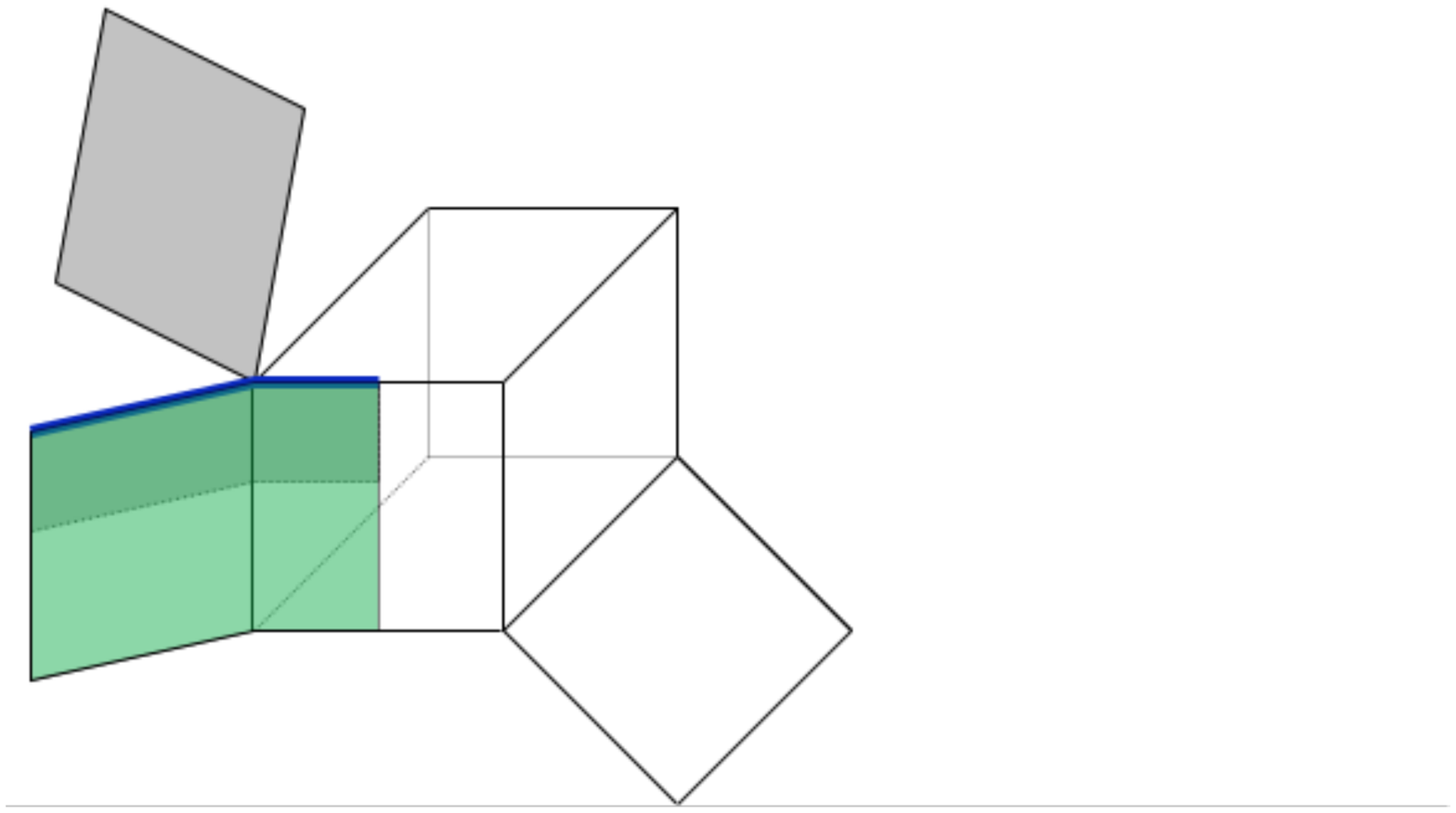}
\end{overpic}
\caption{$Y(\{2\}) \times L_2$}\label{fig:Y(I)2}
\end{minipage}
\begin{minipage}[b]{.45\textwidth}
\centering
\begin{overpic}[trim = 0in 3in 3.55in 2.8in, clip=true, totalheight=0.2\textheight]{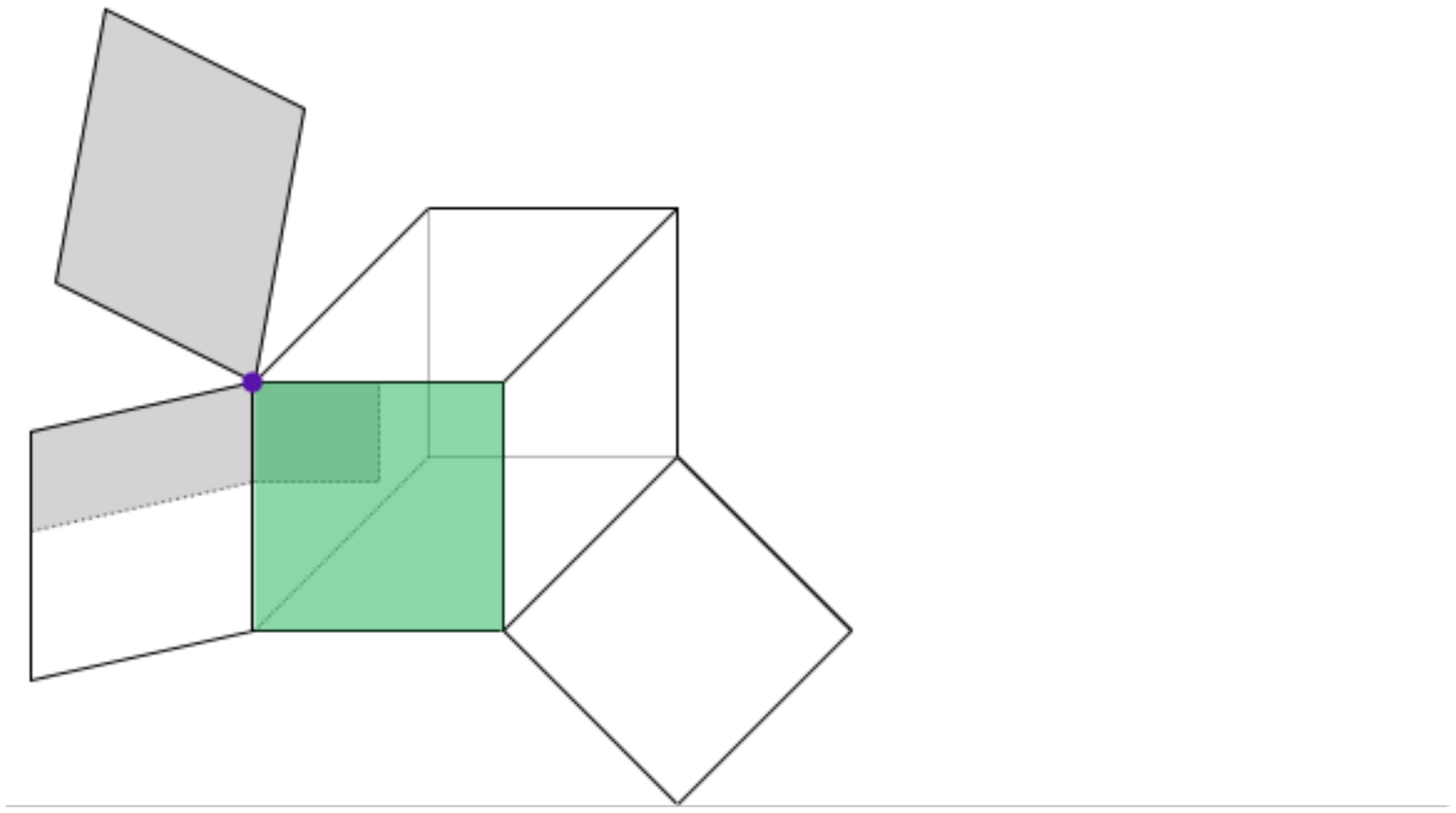}
\end{overpic}
\caption{$Y(\{1,2\}) \times (L_1 \times L_2)$}\label{fig:Y(I)1,2}
\end{minipage}
\end{figure}

To build $\mathfrak q$, we will express $\floor Z$ as the union of $\widehat Z$ and a collection of generalized frames, define $\mathfrak q$ on each generalized frame, and check that the definition is compatible where multiple generalized frames intersect.  Let $z\in\floor Z$ be a $0$--cube.  Either $z\in\widehat Z$, or there is a nonempty set $I\subset\{1,\ldots,k\}$ such that the set of hyperplanes separating $z$ from $\widehat Z$ is contained in $\cup_{i\in I}\mathfrak N'_i$, and each $\mathfrak N'_i$ contains a hyperplane separating $z$ from $\widehat Z$.  If $H\in \mathfrak A(I)\cup\bigcup_{i\in I}\mathfrak N'_i$ is separated from $z$ by a hyperplane $U$, then $U\in \mathfrak A(I)\cup\bigcup_{i\in I}\mathfrak N'_i$, whence we can choose $F(I)$ to contain $z$.  Hence $\floor Z$ is the union of $\widehat Z$ and a finite collection of generalized frames $F(I_1),\ldots,F(I_t)$.

For any $p\in\{1, \dots,t\}$, we have $F(I_p)=Y(I_p)\times\prod_{j\in I_p}L_j$ and we let $\overline Y (I_p) = \image(Y(I_p) \rightarrow Z)$ and let $\overline L_j = \image(L_j\cap\widehat Z \rightarrow Z)$ be the cycle of length $n_{_{N_j}}$ to which $L_j$ maps, for each $j\in I_p$. Note that $Z$ contains $\overline F(I_p) = \overline Y(I_p) \times \prod_{j\in I_P}\overline L_j$ and so we define the quotient map $\mathfrak q_p : F(I_p) \rightarrow Z $ as the product of the above combinatorial quotient maps. Namely, $\mathfrak q_p (y, (r_j)_{j\in I_P}) = (\overline y, (r_j \mod n_{_{N_j}})_{j\in I_p})$ for $y\in Y(I_p)$ and $r_j\in L_j$.

To ensure that $\mathfrak q_p ( F(I_p) \cap F(I_j)) = \mathfrak q_j ( F(I_p) \cap F(I_j))$ for all $i,j\leq t$, it suffices to show that $$F(I_p) \cap F(I_j) := \left(Y(I_p) \times \prod_{k \in I_p} L_k\right) \cap \left(Y(I_j) \times \prod_{\ell \in I_j} L_\ell \right) = \left[Y(I_p) \cap Y(I_j)\right] \times \prod_{k \in I_p \cap I_j} L_k.$$

\noindent This in turn follows from \cite[Proposition 2.5]{CapraceSageev:rank_rigidity}. Hence, the quotient maps $\mathfrak q_p$ are compatible and therefore define a combinatorial quotient map $\mathfrak q: \floor Z \rightarrow Z$ extending the maps $\mathfrak q_p$.

Observe that if $\mathfrak H=\emptyset$, i.e. $K=\floor Z$, then we take $Y=Z$.  By hypothesis, $Z$ admits a local isometry to $S_\Gamma$ and has the desired cardinality.  Moreover, our hypothesis on $g$ ensures that $g\not\in\pi_1Y$, but the map $\mathfrak q$ shows that any closed combinatorial path in $S_\Gamma$ representing $g$ lifts to a (non-closed) path in $Z$, so the proof of the theorem is complete.  Thus we can and shall assume that $\mathfrak H\neq\emptyset$.

\emph{Quotients of $\mathfrak H$-frames:}  To extend $\mathfrak q$ to the rest of $K$, we now describe quotient maps, compatible with the map $\widehat Z \rightarrow Z$, on frames associated to hyperplanes in $\mathfrak H$.  An \emph{isolated $\mathfrak H$-frame} is a frame $(H \cap K) \times L$, where $H \in \mathfrak H$ and $H$ crosses no hyperplane of $\widehat Z$ (and hence crosses no hyperplane of $\floor Z$).  An \emph{interfered $\mathfrak H$-frame} is a frame $(H \cap K) \times L$, where $H \in \mathfrak H$ and $H$ crosses an element of $\mathfrak Z$.  Equivalently, $(H\cap K)\times L$ is interfered if $\gate{\widehat Z}{N(H)}$ contains a $1$--cube and is isolated otherwise.  

Define quotient maps on isolated $\mathfrak H$-frames by the same means as was used for arbitrary frames in \cite{BouRabeeHagenPatel:res_finite_special}: let $(H \cap K) \times L$ be an isolated $\mathfrak H$-frame. Let $\overline H$ be the immersed hyperplane in $S_\Gamma$ to which $H$ is sent by $\widetilde S_\Gamma \rightarrow S_\Gamma$, and let $\overline{H \cap K}$ be the image of $H \cap K$. We form a quotient $Y_H = \overline{H \cap K} \times L$ of every isolated $\mathfrak H$-frame $(H \cap K) \times L$.

Now we define the quotients of interfered $\mathfrak H$-frames.  Let $\widehat A=\gate{\widehat Z}{N(H)}$ and let $A$ be image of $\widehat A$ under $\widehat Z \rightarrow Z$. There is a local isometry $A \rightarrow S_\Gamma$, to which we apply canonical completion to produce a finite cover $\dddot S_\Gamma\rightarrow S_\Gamma$ where $A$ embeds. By Lemma 2.8 of \cite{BouRabeeHagenPatel:res_finite_special}, $\deg(\dddot S_\Gamma\rightarrow S_\Gamma)= \left|\dddot S_\Gamma^{(0)}\right| = \left| A^{(0)}\right| \leq \left|Z^{(0)}\right|$.  Let $\overline{H \cap K} = \image(H \cap K \rightarrow \dddot S_\Gamma)$, and map the interfered $\mathfrak H$-frame $(H \cap K) \times L$ to $Y_H = \overline{H \cap K} \times L$.

\emph{Constructing $Y$:}  We  now construct a compact cube complex $Y'$ from $Z$ and the various quotients $Y_H$.  A hyperplane $W$ in $K$ separates $H_1$ from $\widehat Z$ only if $W\in\mathfrak N$. Each $\mathfrak{H}$-hyperplane frame has the form $(H_i \cap K) \times L_i = (H_i \cap K) \times [0, \, m_i]$, parametrized so that $(H_i \cap K) \times \{0\}$ is the closest combinatorial hyperplane in the frame to $\widehat Z$. We form $Y'(1)$ by gluing $Y_{H_1}$ to $Z$ along the image of $\gate{(H_1 \cap K) \times \{0\}}{\widehat Z}$, enabled by the fact that the quotients of interfered $\mathfrak H$-frames are compatible with $\widehat Z \rightarrow Z$. In a similar manner, form $Y'(i)$ from $Y'(i-1)$ and $Y_{H_i}$ by identifying the image of $(H_{i-1} \cap K) \times \{m_{i-1}\} \cap (H_i \cap K) \times \{0\}$ in $Y_{H_{i-1}} \subset Y'(i-1)$ with its image in $Y_{H_{i}}$. Let $Y' = Y'(k)$.

Let $K'=\floor Z \cup \left[ \bigcup_{H_i \in \mathfrak H} (H_i \cap K) \times L_i \right]$. Since $H_i \cap H_j = \emptyset$ for $i \neq j$, there exists a map $(K', \tilde x) \rightarrow (Y', x)$ and a map $(Y', x) \rightarrow (S_\Gamma, x)$ such that the composition is precisely the restriction to $K'$ of the covering map $(\widetilde S_\Gamma, \tilde x) \rightarrow (S_\Gamma, x)$.

If $Y'\rightarrow S_\Gamma$ fails to be a local isometry, then there exists $i$ and nontrivial open cubes $e\subset\overline{H_{i-1} \cap K} \times \{m_{i-1}\}$ (or $Z$ if $i=1$) and $c\subset\overline{H_i \cap K} \times \{0\}$ such that $S_\Gamma$ contains an open cube $\bar e\times\bar c$, where $\bar e,\bar c$ are the images of $e,c$ under $\widetilde S_\Gamma \rightarrow S_\Gamma$, respectively.  Moreover, since $\gate{H_i \cap K}{\widehat{Z}} \subseteq \gate{H_{i-1} \cap K}{\widehat{Z}}$, we can assume that $\bar c$ is disjoint from each immersed hyperplane of $S_\Gamma$ crossing $Z$.  Hence the closure $\closure{\bar c}$ is a standard torus.  Glue $\closure{\bar e}\times \closure{\bar c}$ to $Y'$, if necessary, in the obvious way.  Note that this gluing adds no new $0$--cubes to $Y'$.  Indeed, every $0$--cube of $\closure {\bar e} \times \closure {\bar c}$ is identified with an existing $0$--cube of $Y'$ lying in $\overline {H_{i-1} \cap K} \times \{ m_{i-1}\}$. Adding $\closure{\bar e}\times \closure{\bar c}$ also preserves the existence of a local injection from our cube complex to $S_\Gamma$. Either this new complex admits a local isometry to $S_\Gamma$, or there is a missing cube of the form $\bar e \times \bar c$ where $\closure{\bar c}$ is a standard torus and $\bar e$ lies in $Y'$. We add cubes of this type until we have no missing corners. That the process terminates in a local isometry with compact domain $Y$ is a consequence of the following facts: at each stage, every missing cube has the form $\bar e \times \bar c$ where $\bar e$ lies in $Y'$ and $\closure{\bar c}$ is a standard torus, so the number of 0--cubes remains unchanged; each gluing preserves the existence of a local injection to $S_\Gamma$; each gluing increases the number of positive dimensional cubes containing some $0$--cube; cubes that we add are images of cubes in $K$, which is compact.

Note that there exists a combinatorial path $\gamma$ in $K'$ joining $\tilde x$ to $g \tilde x$.  It follows from the existence of $\gamma$ that the convex hull of $K'$ is precisely equal to $K$. Hence, there exists a based cubical map $(K, \tilde x) \rightarrow (Y, x) \rightarrow (S_\Gamma, x)$, so that the composition is the restriction of the covering map $(\widetilde S_\Gamma, \tilde x) \rightarrow (S_\Gamma, x)$. Therefore, any closed path in $S_\Gamma$ representing $g$ lifts to a non-closed path at $x$ in $Y$. It is easily verified that the number of $0$--cubes in $Y$ is bounded by $|Z^{(0)}| (m_1 + \cdots + m_k)$ where each $m_i$ is the length of $L_i$, and hence $|Y^{(0)}| \leq |Z^{(0)}| (|g| +1)$. Thus, $Y$ is the desired cube complex.
\end{proof}
\begin{rem}\label{rem:free_group_case}
When $\dimension S_\Gamma=1$, arguing as above shows that $Y$ can be chosen so that $|Y^{(0)}|\leq|Z^{(0)}| + |g|$.  Hence, if $F$ is freely generated by $\mathcal S$, with $|\mathcal S|=r$, then $\sepgrowth^{\mathcal Q_K}_{F,\mathcal S}(Q,n)\leq (2r)^{K}+n.$
\end{rem}

\bibliographystyle{acm}
\bibliography{res_state}
\end{document}